\documentclass[article,reqno]{amsart}
\usepackage[mathscr]{eucal}
\usepackage{amsmath,amsfonts,amsthm,amssymb,mathrsfs,url}
\usepackage{array}

\usepackage{hyperref}
\hypersetup{
    bookmarks=true,         % show bookmarks bar?
    unicode=false,          % non-Latin characters in Acrobat�s bookmarks
    pdftoolbar=true,        % show Acrobat�s toolbar?
    pdfmenubar=true,        % show Acrobat�s menu?
    pdffitwindow=false,     % window fit to page when opened
    pdfstartview={FitH},    % fits the width of the page to the window
    pdftitle={Lorentzian polynomials on cones and the Heron-Rota-Welsh conjecture},    % title
    pdfauthor={Author},     % author
    pdfsubject={Subject},   % subject of the document
    pdfcreator={Creator},   % creator of the document
    pdfproducer={Producer}, % producer of the document
    pdfkeywords={keyword1} {key2} {key3}, % list of keywords
    pdfnewwindow=true,      % links in new window
    colorlinks=true,       % false: boxed links; true: colored links
    linkcolor=blue,          % color of internal links
    citecolor=blue,        % color of links to bibliography
    filecolor=blue,      % color of file links
    urlcolor=black    % color of external links
    }

\theoremstyle{plain}
   \newtheorem{theorem}{Theorem}[section]
   \newtheorem{proposition}[theorem]{Proposition}
   \newtheorem{lemma}[theorem]{Lemma}

\theoremstyle{definition}
   \newtheorem{definition}{Definition}[section]
   \newtheorem{example}{Example}[section]

\theoremstyle{remark}
   \newtheorem{remark}[theorem]{Remark}

\author[P.~Br\"and\'en]{Petter Br\"and\'en}
\address{Department of Mathematics, KTH Royal Institute of Technology, SE-100 44 Stockholm,
Sweden}
\email{pbranden@kth.se}

\author[J.~Leake]{Jonathan Leake}
\address{Department of Mathematics, Technische Universit\"at Berlin, Germany}
\email{jonathan@jleake.com}

\numberwithin{equation}{section}
\newcommand{\xx}{\mathbf{x}}
\newcommand{\yy}{\mathbf{y}}
\newcommand{\zz}{\mathbf{z}}

\newcommand{\vv}{\mathbf{v}}
\newcommand{\ttt}{\mathbf{t}}
\newcommand{\ff}{\pol}

\newcommand{\ww}{\mathbf{w}}
\newcommand{\uu}{\mathbf{u}}

\newcommand{\MOD}{\mathscr{M}}
\newcommand{\SMOD}{\mathscr{C}}
\newcommand{\PP}{\mathscr{B}}
\newcommand{\PPP}{\mathscr{P}}

\newcommand{\EE}{\mathscr{E}}

\newcommand{\LL}{\mathscr{L}}

\newcommand{\MM}{\mathrm{M}}

\newcommand{\RR}{\mathbb{R}}

\newcommand{\CCC}{\mathscr{C}}
\newcommand{\WWW}{\mathscr{W}}

\def\newop#1{\expandafter\def\csname #1\endcsname{\mathop{\rm
#1}\nolimits}}

\newop{per}
\newop{diag}
\newop{supp}
\newop{rk}
\newop{Deg}
\newop{span}
\newop{rank}
\newop{Sym}
\newop{sign}
\newop{Int}
\newop{disc}
\newop{mult}
\newop{tr}
\newop{vol}
\newop{pol}
\newop{MAP}

\title[Lorentzian polynomials on cones]{Lorentzian polynomials on cones and the Heron-Rota-Welsh conjecture}

\begin{document}
\begin{abstract}We give a short proof of the log-concavity of the coefficients of the reduced characteristic polynomial of a matroid. The proof uses an extension of the theory of Lorentzian polynomials to convex cones, and reproves the Hodge-Riemann relations of degree one for the Chow ring of a matroid.
\end{abstract}
\maketitle
\thispagestyle{empty}

\section{Introduction}
Over the past decade, methods originating in algebraic geometry have been developed  to solve long-standing conjectures on unimodality and log-concavity in matroid theory, see \cite{Baker}. In \cite{AHK}, a Hodge theory of matroids was developed by Adiprasito, Huh and Katz to prove the Heron-Rota-Welsh conjecture on the log-concavity of the characteristic polynomial of a matroid. 

A different approach to log-concavity problems in matroid theory originates in the works of Choe, Oxley, Sokal, Wagner, Gurvits and the first author, see \cite{WSur}. This approach uses convexity properties of multivariate polynomials rather than Hodge theory. Recently breakthroughs  in this approach were made by Huh and the first author \cite{BH}, and by Anari, Liu, Oveis Gharan and Vinzant \cite{ALOV}. The theory of Lorentzian polynomials  was developed in \cite{ALOV, BH, Gu}, and used in \cite{ALOV,BH} to prove the strongest of Mason's conjectures on the log-concavity of the  number of independent sets of a matroid. 

In \cite{BES}, Backman, Eur and Simpson combined Lorentzian polynomials and Hodge theory to give a shorter proof of the Hodge-Riemann relations of degree one for the Chow ring of a matroid. The proof in Section 6 of \cite{BES} assumes Poincar\'e duality for the Chow ring of a matroid, which is proved in \cite{AHK} and in \cite[Section 4]{BES}. However, until now a purely “polynomial proof” of the Heron-Rota-Welsh conjecture -- using the theory of Lorentzian polynomials and avoiding Hodge theory altogether -- has been missing. The main purpose of this paper is to give such a proof.

In Section \ref{Lor-cones} we extend the theory of Lorentzian polynomials so that it applies to cones other than the positive orthant (see Proposition~\ref{engine}), and use this extension to give a very short proof of the Heron-Rota-Welsh conjecture in Sections \ref{posetsec} and  \ref{lc} (see Theorem~\ref{HRWconj}), which does not rely on Hodge theory. In fact we give a self-contained proof of the Hodge-Riemann relations of degree one for the Chow ring of a matroid, see Theorem~\ref{mainChow} and Theorem~\ref{chowVol}.

\section{Lorentzian polynomials on cones}\label{Lor-cones}
Let $\partial_i$ (or $\partial_{x_i}$) denote the partial derivative with respect to $x_i$, and for $\uu=(u_1, \ldots, u_n) \in \RR^n$ let 
$D_\uu= u_1\partial_1 + \cdots + u_n \partial_n$.  

\begin{definition}\label{C-def}
Let $\CCC$ be an open convex cone in $\RR^n$. A homogeneous polynomial $f \in \RR[x_1,\ldots, x_n]$ of degree $d$ is called $\CCC$-\emph{Lorentzian} if for all 
$\vv_1, \ldots, \vv_{d} \in \CCC$, 
\begin{itemize}
\item[(P)] $D_{\vv_1}\cdots D_{\vv_{d}} f>0$, and 
\item[(H)] the symmetric bilinear form 
$$
(\xx, \yy) \mapsto D_{\xx}D_{\yy} D_{\vv_3}\cdots D_{\vv_{d}} f
$$
has exactly one positive eigenvalue. 
\end{itemize}
\end{definition}
 By convention, we also say that the identically zero polynomial is $\CCC$-Lorentzian. 
 
Recall that the \emph{Hessian} of $f$ at $\xx$ is the matrix $\nabla^2 f (\xx) = (\partial_i\partial_j f(\xx))_{i,j=1}^n$. Hence (H) asserts that the Hessian of   $D_{\vv_3}\cdots D_{\vv_{d}} f$ has exactly one positive eigenvalue. 

\begin{remark}\label{altdef}
It follows from \cite[Thm. 2.25]{BH} that Definition \ref{C-def} is equivalent to that for all positive integers $m$ and for all $\vv_1, \ldots, \vv_{m} \in \CCC$, the polynomial 
$$
(y_1,\ldots, y_m) \mapsto f(y_1\vv_1+ \cdots+ y_m\vv_m)
$$
 is Lorentzian (in the sense of \cite{BH}) and has positive coefficients only. 
 It also follows that Lorentzian polynomials are the same as $\RR_{>0}^n$-Lorentzian polynomials. 
 \end{remark}

From Remark \ref{altdef} and \cite[Cor. 2.32]{BH} we deduce

\begin{proposition}\label{prod}
Suppose $f$ and $g$ are $\CCC$-Lorentzian, then so is $fg$. 
\end{proposition}

A matrix $A=(a_{ij})_{i,j=1}^n$ whose off-diagonal entries are nonnegative is called \emph{irreducible} if for all distinct $i,j$ there is a sequence $i=i_0,i_1,i_2,\ldots, i_\ell=j$  such that $i_{k-1}\neq i_{k}$ for all $1 \leq k \leq \ell$, and 
$
a_{i_0 i_1} a_{i_1 i_2} \cdots a_{i_{\ell-1}i_{\ell}}>0.
$
By translating such a matrix with a positive multiple of the identity matrix, the Perron-Frobenius theory \cite[Chapter 1]{BP} guarantees that $A$ has a unique eigenvector (up to multiplication by positive scalars) whose entries are all positive. Moreover the corresponding eigenvalue is simple and is the largest eigenvalue of $A$. 

If $A$ and $B$ are symmetric matrices of the same size, we write $A \preceq B$ if $B-A$ is positive semidefinite.  

The following lemma is,  in essence, taken from \cite[Prop. 3]{AFI}. 
\begin{lemma}\label{indlemma}
Let $f$ be a homogeneous polynomial of degree $d \geq 3$, and let $\xx \in \RR_{>0}^n$. If
\begin{enumerate}
\item  $\partial_i f(\xx)>0$ for all $i$, and 
\item the Hessian  of $\partial_i f$ at $\xx$ has exactly one positive eigenvalue for all $i$, and 
\item the Hessian of $f$ at $\xx$ is irreducible, and its off-diagonal entries are nonnegative, 
\end{enumerate}
then the Hessian of $f$ at $\xx$ has exactly one positive eigenvalue. 
\end{lemma}
\begin{proof}
 If $g$ is a homogeneous polynomial of degree $d$ and $g(\xx)>0$, then the following three statements are equivalent 
\begin{itemize}
\item[(a)] the Hessian of $g$ at $\xx$ has exactly one positive eigenvalue,
\item[(b)]   the Hessian of $g^{1/d}$  is negative semidefinite at $\xx$, 
\item[(c)] the matrix $d\cdot g \cdot \nabla^2g - (d-1)\cdot \nabla g (\nabla g)^T$ is negative semidefinite at $\xx$,
\end{itemize}
see e.g. \cite[Prop. 2.33]{BH}. 

Suppose $\xx$ and $f$ are as in the hypotheses of the lemma. 
Then, by (c), 
$$
(d-1)\cdot \partial_i f \cdot \nabla^2 \partial_i f  \preceq (d-2) \cdot \nabla \partial_i f (\nabla \partial_i f)^T.
$$
Euler's identity, 
$
d \cdot f(\xx) = \sum_{i=1}^n x_i \cdot \partial_i f(\xx),
$
 yields 
$$
(d-2)\cdot \nabla^2 f = \sum_{i=1}^n x_i \nabla^2  \partial_i f  \preceq  \sum_{i=1}^n \frac {x_i}{\partial_i f} \frac {(d-2)}{(d-1)} \nabla \partial_i f (\nabla \partial_i f)^T.
$$
Rewrite the above inequality as
$
(d-1) \cdot \nabla^2 f \preceq (\nabla^2 f)  \Lambda (\nabla^2 f), 
$ 
where $\Lambda$ is the diagonal matrix $\diag({x_1}/{\partial_1 f}, \ldots, {x_n}/{\partial_n f})$. For the matrix $B= \Lambda^{1/2} (\nabla^2 f) \Lambda^{1/2}$,  this implies 
$
B^2 -(d-1)B \succeq 0.
$
Hence no eigenvalue of $B$ lies in the open interval $(0,d-1)$. The matrix $B$ is irreducible and has nonnegative off-diagonal entries, so the Perron-Frobenius theorem applies to $B$. 
Notice that $\Lambda^{-1/2} \xx$ is a positive eigenvector of $B$, and the corresponding eigenvalue is $d-1$. Hence   $d-1$ is the unique largest eigenvalue of $B$ afforded by  the Perron-Frobenius theorem. We conclude that $B$, and thus also $\nabla^2 f(\xx)$, has exactly one positive eigenvalue.
\end{proof}

Recall that the \emph{lineality space} of an open convex  cone $\CCC$ in $\RR^n$ is $L_\CCC=\overline{\CCC}\cap -\overline{\CCC}$, i.e., the largest linear space contained in the closure of $\CCC$. We say that $\CCC$ is \emph{effective} if 
$\CCC= \CCC \cap \RR_{>0}^n + L_\CCC$.

\begin{proposition}\label{engine}
Let $f\in \RR[x_1,\ldots, x_n]$ be a homogeneous polynomial of degree $d \geq 3$, and let $\CCC$ be an open, convex and effective cone in $\RR^n$. If 
\begin{enumerate}
\item $f(\xx+\ww)=f(\xx)$ for all $\xx \in \RR^n$ and $\ww \in L_\CCC$, and 
\item $D_{\vv_1}\cdots D_{\vv_d}f >0$ for all $\vv_1, \ldots, \vv_d \in \CCC$, and 
\item  the Hessian of $D_{\vv_1} \cdots D_{\vv_{d-2}}f$ is irreducible and its off-diagonal entries are nonnegative for all  $\vv_1,\ldots,  \vv_{d-2} \in \CCC$,  and 
\item $\partial_i f$ is $\CCC$-Lorentzian for all $i$, 
\end{enumerate}
then $f$ is $\CCC$-Lorentzian.
\end{proposition}

\begin{proof}
Let $\vv_1, \ldots, \vv_{d-3} \in \CCC$, and consider the cubic  $g= D_{\vv_1}\cdots D_{\vv_{d-3}} f$. 
Since $\partial_i f$ is $\CCC$-Lorentzian, it follows from Definition~\ref{C-def} that so is $\partial_i g$. By choosing $\vv_{d-2}=\xx \in \CCC$, it follows from (3) that 
the Hessian $\nabla^2g(\xx)$ is irreducible and its off-diagonal entries are nonnegative. Since 
$$
g(\xx) = \frac {\partial^{d-3}}{ \partial t_1 \cdots \partial t_{d-3}} f\left(\xx + \sum_{i=1}^{d-3} t_i \vv_i\right), \ \ \ t_1=\cdots=t_{d-3}=0, 
$$
it follows from (1) that $g(\xx+\ww)=g(\xx)$ for all $\xx \in \RR^n$ and $\ww \in L_\CCC$.  Since $\CCC$ is effective, we may assume  $\xx \in \CCC \cap \RR_{>0}^n$. Lemma \ref{indlemma} then implies that the Hessian of $g$ at $\xx$ has exactly one positive eigenvalue. 

The lemma now follows since the Hessian of $D_{\vv_1}\cdots D_{\vv_{d-3}} f$ at $\xx$ is equal to the Hessian of $D_{\vv_1}\cdots D_{\vv_{d-3}}D_\xx f$. 
\end{proof}

\section{Polynomials associated to graded sub-posets of Boolean lattices}\label{posetsec}
Let $\PP(E)= \{S : S \subseteq E\}$ denote the \emph{Boolean lattice} of subsets of a finite set $E$. In this section $\PPP=(X, \leq)$ will be any  sub-poset of $\PP(E)$, for which each finite closed interval $[K,L]_\PPP  = \{F \in \PPP:  K \leq F \leq L\}$ in $\PPP$ is graded. We write $K \prec L$ if $K<L$ and there is no $F \in \PPP$ for which $K<F<L$.

 For $\varnothing \subseteq K \subset L$, let $\EE_K^L =\{ (y_S)_{K \subset S \subset L} : y_S \in \RR\}= \RR^m$, where $m= 2^{|L\setminus K|}-2$. Denote by $\MOD_K^L$, the subspace of \emph{modular} elements $\yy$ in $\EE_K^L$, i.e.,
$$
y_S+ y_T= y_{S\cap T}+y_{S\cup T}, \ \ \mbox{ for all } S,T, 
$$
where $y_K=y_L=0$. 
 It follows that $\yy \in \MOD_K^L$ if and only if there are real numbers $y_e$, $e \in L\setminus K$, for which $\sum_{e \in L\setminus K} y_e=0$ and 
$
y_S= \sum_{e \in S\setminus K}y_e,
$
for all $K \subset S \subset L$.

 If $K \subseteq F \subset G \subseteq L$, define a linear projection $\pi_F^G : \EE_K^L \to \EE_F^G$  by 
$$
\pi_F^G(\ttt) = \left(t_S -t_G \frac{|S\setminus F|}{|G\setminus F|}-t_F\frac{|G \setminus S|}{|G \setminus F|} \right)_{F \subset S \subset G}, 
$$
where $t_K=t_L=0$. 

Let $r(K,L)$ be the rank of the interval $[K,L]_\PPP$, and let $d(K,L)=r(K,L)-1$. We define a polynomial $\ff_K^L(\ttt)$, of degree $d(K,L)$, in the variables $\{t_F : K <  F < L\}$, for each $K<L$ in $\PPP$. In Section \ref{chowder} we will prove that $\ff_K^L(\ttt)$ is the volume polynomial of the Chow ring of a matroid, as defined in \cite{AHK}. Notice that while $\ff_K^L(\ttt)$ is defined on variables indexed by $F \in \PPP$ such that $K < F < L$, we will often consider it as a polynomials in the larger set of variables $\{t_S : K\subset S \subset L\}$. 

\begin{definition}The polynomial $\ff_K^L(\ttt)$,  associated to  $K<L$ in $\PPP$ is defined recursively as follows.  

If $d(K,L)=0$, then $\ff_K^L(\ttt)=1$. If $d(K,L) \geq 1$, then
\begin{equation}\label{f-def}
d(K,L) \cdot \ff_K^L (\ttt) = \sum_{K<F<L} t_F \cdot  \ff_K^F ( \pi_K^F(\ttt))\cdot  \ff_F^L (\pi_F^L(\ttt)). 
\end{equation}
\end{definition}
 
\begin{example}
If  $d(K,L)=1$, then 
$
\ff_K^L (\ttt) = \sum_{K\prec F \prec L}t_F.  
$
If $d(K, L) = 2$, then 
\begin{equation}\label{case2}
2 \cdot \ff_K^L(\ttt)=\sum_{K\prec F\prec G\prec L} \left(2\cdot t_Ft_G -t_F^2 \cdot \frac {|L\setminus G|}{|L \setminus F|} - t_G^2 \cdot \frac {|F \setminus K|}{|G\setminus K|}\right).
\end{equation}
\end{example}
 
Let $\SMOD_K^L$ denote the open convex cone in $\EE_K^L$ consisting of all \emph{strictly sub-modular} $\yy \in \EE_K^L$ i.e., 
$$
y_S+ y_T>y_{S\cap T}+y_{S\cup T}
$$
for all incomparable $S$ and $T$, where $y_K=y_L=0$. Hence the lineality space of $\SMOD_K^L$ is $\MOD_K^L$.

\begin{lemma}\label{pos-pair-mod}
If $|L\setminus K| \geq 2$, then $\SMOD_K^L$ is effective.
\end{lemma}

\begin{proof}
It is plain to see that  $\vv = (|S\setminus K|\cdot |L\setminus S|)_{K\subset S \subset L} \in \SMOD_K^L \cap \RR_{> 0}^{(K,L)}$,  where $(K,L)= \{S : K\subset S \subset L\}$. 
Suppose $\yy \in \SMOD_K^L$.  Then $\zz:=\yy -\epsilon \vv \in \SMOD_K^L$, for $\epsilon>0$ sufficiently small. By e.g. \cite[Prop.~4.4]{Murota}, there exists  $\ww \in \MOD_K^L$ such that $\zz+\ww \in \RR_{\geq 0}^{(K,L)}$. But then $\yy+\ww = \zz +\ww+\epsilon \vv \in  \RR_{> 0}^{(K,L)}$, as desired. 
\end{proof}

\begin{lemma}\label{Euler-cons}
Suppose $f \in \RR[t_1, \ldots, t_n]$ is a homogeneous polynomial of degree $d$, and that 
$
df = \sum_{i=1}^n t_i Q_i,
$
where $Q_1, \ldots, Q_n$ are homogeneous polynomials of degree $d-1$ for which 
$
\partial_i Q_j = \partial_j Q_i$, for all $i,j$. 
Then $Q_i= \partial _i f$, for all $i$. 
\end{lemma}
\begin{proof}
By Euler's identity,
$$d \partial_j f = \partial_j(t_jQ_j) -t_j\partial_jQ_j+ \sum_{i=1}^n t_i\partial_jQ_i  = Q_j + \sum_{i=1}^n t_i\partial_iQ_j = Q_j+(d-1)Q_j= dQ_j.$$
\end{proof}

\begin{lemma}\label{derivativen}
If $K<F<L$, then 
\begin{equation}\label{splitt}
\frac {\partial} {\partial t_F} \ff_K^L(\ttt)=  \ff_K^F ( \pi_K^F(\ttt))\cdot  \ff_F^L (\pi_F^L(\ttt)). 
\end{equation}
\end{lemma} 
\begin{proof}
By induction it follows from \eqref{f-def}  that 
$
\partial_{t_F} \partial_{t_G} \ff_K^L(\ttt)=0,
$
unless $F$ and $G$ are comparable. We now prove \eqref{splitt} by induction over $d=d(K,L)$, the case when $d=1$ being clear. Suppose $d>1$, and let 
$
Q_F(\ttt)= \ff_K^F ( \pi_K^F(\ttt))\cdot  \ff_F^L (\pi_F^L(\ttt)). 
$
By Lemma \ref{Euler-cons}, it remains  to prove 
$
\partial_{t_F}Q_G(\ttt) = \partial_{t_G}Q_F(\ttt),
$
for all $G<F$. By induction
$$
\frac {\partial} {\partial t_F}Q_G(\ttt) = \ff_K^G ( \pi_K^G(\ttt))\cdot  \ff_G^F (\pi_G^F\pi_G^L(\ttt))\cdot \ff_F^L (\pi_F^L\pi_G^L(\ttt)), 
$$
and 
$$
\frac {\partial} {\partial t_G}Q_F(\ttt) = \ff_K^G ( \pi_K^G\pi_K^F(\ttt))\cdot  \ff_G^F (\pi_G^F\pi_K^F(\ttt))\cdot \ff_F^L (\pi_F^L(\ttt)). 
$$
Now $
\partial_{t_F}Q_G(\ttt) = \partial_{t_G}Q_F(\ttt)
$ follows since $\pi_G^F\pi_K^L = \pi_G^F$ whenever $K\subseteq G \subset F \subseteq L$.
\end{proof}
The proof of the next lemma is left to the reader. 
\begin{lemma}\label{piserve}
If $K\subseteq G \subset F \subseteq L$, then 
$\pi_G^F : \MOD_K^L \to \MOD_G^F  \mbox{ and }  \pi_G^F : \SMOD_K^L \to \SMOD_G^F.$
\end{lemma}

We call a sub-poset $\PPP$ of $\PP(E)$  \emph{balanced} if for each $K<L$ in $\PPP$ such that $d(K,L)=1$,
$$
|\{ K<F<L : F \ni i\}|= |\{ K<F<L : F \ni j\}|, \ \ \ \mbox{ for all } i,j \in L \setminus K. 
$$
In particular, $\PPP$ is balanced if $\{ A\setminus K\}_{K\prec A \leq L}$ partitions $L \setminus K$, whenever $d(K,L)=1$.  If so, we say that  $\PPP$ is $1$-\emph{balanced}. 

\begin{lemma}\label{lineality}
Suppose $\PPP$ is balanced, and let $K<L$. 
Then 
$
\ff_K^L(\xx+\yy)= \ff_K^L(\xx)
$, 
for all $\xx \in \EE_K^L$ and $\yy \in \MOD_K^L$. 
\end{lemma}

\begin{proof}
The proof is by induction on $d(K,L)$. Suppose $d(K,L)=1$. 
Since $\yy$ is modular and $\ff_K^L$ is linear,  
$$
\ff_K^L(\xx+\yy)-\ff_K^L(\xx)= \sum_{i \in L\setminus K} y_{i} \cdot |\{ K<F<L : F \ni i\}|.
$$
The case when $d=1$ thus follows since $\PPP$ is balanced and $\sum_{i \in L\setminus K} y_{i}=0$. 

Consider the space $\WWW$ of all homogeneous polynomials $f(\ttt)$ in $\RR[t_S : F\subset S \subset L]$ for which 
$
f(\ttt +\yy)= f(\ttt)$,   for all  $\yy \in \MOD_K^L$.
Then $f \in \WWW$ if and only if $\partial^\alpha f(\yy) =0$ for all $\yy \in \WWW$ and $|\alpha|<\deg (f)$, since $f(\ttt+\yy)=f(\ttt)$ is supposed to be homogeneous.  By induction it follows from \eqref{f-def} that $\ff_K^L(\yy)=0$ for all $\yy \in \MOD_K^L$, whenever $d(K,L) \geq 1$. 
Suppose $d(K,L)>1$.  Then $\partial_{t_F} \ff_K^L \in \WWW$ by Lemma \ref{derivativen}, Lemma \ref{piserve} and induction. Hence $\partial^\alpha \ff_K^L(\yy)=0$ for all $|\alpha|< d(K,L)$, so that $\ff_K^L \in \WWW$. 
\end{proof}

We call $\PPP$ \emph{interval connected} if for each $K<L$  with $d(K,L) \geq 2$ and $F, G \in (K,L)_\PPP=\{H \in \PPP : K<H<L\}$, there exists elements 
$F_i \in (K,L)_\PPP$ such that $F=F_0 \lessgtr F_1 \lessgtr F_2 \lessgtr \cdots \lessgtr F_k =G, $ where $F_i \lessgtr F_{i+1}$ means 
$F_i <F_{i+1}$ or $F_i > F_{i+1}$. 

\begin{lemma}\label{positivity}
Suppose $\PPP$ is balanced, and that $K<L$.  
If $\vv_1, \ldots, \vv_d \in \SMOD_K^L$, then 
$
D_{\vv_1}\cdots D_{\vv_d} \ff_K^L >0.
$
Moreover if $d(K,L) \geq 2$ and $\PPP$ is interval connected, then the  Hessian of the quadratic $D_{\vv_1} \cdots D_{\vv_{d-2}}\ff_K^L$ is irreducible, and its off-diagonal entries are nonnegative. 
 \end{lemma}
 
\begin{proof}
We start by proving the first assertion by induction on $d=d(K,L)$. We need to prove that all coefficients of the polynomial 
$\ff_K^L (s_1\vv_1+ \cdots+s_d\vv_d)$ are positive. By Lemmas \ref{pos-pair-mod} and \ref{lineality}, we may assume that all entries of $\vv_i \in  \SMOD_K^L$ are positive for all $i$. The first assertion now follows by induction  using \eqref{f-def} and Lemma \ref{piserve}.

For $K<F_1<F_2<L$, let
$$
g(\ttt)=\frac {\partial^2} {\partial t_{F_1}\partial t_{F_2}} \ff_K^L(\ttt)=  \ff_K^{F_1} (\pi_K^{F_1}(\ttt))\cdot  \ff_{F_1}^{F_2} (\pi_{F_1}^{F_2}(\ttt)) \cdot  \ff_{F_2}^L (\pi_{F_2}^L(\ttt)). 
$$
As above, it follows that all coefficients of $g(s_1\vv_1+\cdots+ s_{d-2}\vv_{d-2})$ are positive, whenever $\vv_1, \ldots, \vv_{d-2} \in \SMOD_K^L$. Hence 
$
\partial_{t_{F_1}}\partial_{t_{F_2}} D_{\vv_1} \cdots D_{\vv_{d-2}} \ff_K^L >0. 
$
Also the entries corresponding to non-comparable $F_1$ and $F_2$ are zero. Since $\PPP$ is interval connected it follows that 
the  Hessian of $D_{\vv_1} \cdots D_{\vv_{d-2}}\ff_K^L$ is irreducible.  
\end{proof}

Recall, see e.g. \cite{Oxley},  that a sub-poset $\PPP$ of $\PP(E)$ is the \emph{lattice of flats of a matroid} on $E$ if and only if 
\begin{itemize}
\item[(F1)]  $E \in \PPP$, 
\item[(F2)] If $F,G \in \PPP$, then $F \cap G \in \PPP$,  
\item[(F3)] For each $K$ in $\PPP$, $\{A\setminus K\}_{K \prec A}$ partitions $E \setminus K$. 
\end{itemize}
Hence, the lattice of flats of a matroid is $1$-balanced. Also lattices of flats of matroids are graded and semimodular, see \cite[Chapter 1.7]{Oxley}. 
Recall \cite[Prop. 3.3.2]{stanley} that a finite lattice $\LL$ is \emph{semimodular} if and only if for all $a,b \in \LL$, 
$$
\mbox{if $a$ and $b$ cover $a \wedge  b$, then $a \vee b$ covers $a$ and $b$.}
$$
From semimodularity it follows that lattices of flats of matroids are interval connected. From $(\mathrm{F1})$--$(\mathrm{F3})$ it follows that each closed interval of the lattice flats of a matroid is again the lattice of flats of a matroid. 
 
\begin{theorem}\label{mainChow}
If $[K,L]_\PPP$ is the lattice of flats of a matroid, then $\ff_K^L$ is $\SMOD_K^L$-Lorentzian.  
\end{theorem}
\begin{proof}
We shall apply Proposition \ref{engine} for $\SMOD=\SMOD_K^L$ and $\WWW=\MOD_K^L$. The proof is by induction on $d(K,L)$. The case when $d(K,L) \leq 1$ is clear. Also, by Lemma \ref{lineality} and  Lemma \ref{positivity} and the discussion preceding the lemma, (1)--(3) of Proposition \ref{engine} are satisfied.  Suppose $d(K,L) \geq 2$. By Lemma \ref{derivativen}, Lemma \ref{piserve}, Proposition \ref{prod} and induction, $\partial_{t_F} \ff_K^L$ is $\SMOD_K^L$-Lorentzian for each $F \in (K,L)_\PPP$. Hence it remains to prove the case when $d(K,L)=2$. 

Suppose $d(K,L)=2$. We need to prove that the Hessian of $\ff_K^L$ has exactly one positive eigenvalue. Denote flats of rank one in $[K,L]_\PPP$ by $F$, and flats of rank two by $G$. 

Since $[K,L]_\PPP$ is $1$-balanced,
$$
\sum_{F:F<G} \frac {|F\setminus K|}{|G\setminus K|} = 1 \ \ \mbox{ and } \ \ \sum_{G: G>F} \frac {|L\setminus G|} {|L \setminus F|} = |\{G: G >F\}| -1. 
$$
We deduce from \eqref{case2} that $2\ff_K^L=
2 \sum_{F<G} t_Ft_G -  \sum_{F<G} t_F^2+ \sum_{F} t_F^2-\sum_{G} t_G^2.
$ 
Notice that 
$$
\sum_{G} \left(t_G-\sum_{F<G}t_F\right)^2= \sum_Gt_G^2 -2 \sum_{F<G} t_Ft_G+ \sum_{F<G}t_F^2+ \sum_G\sum_{\stackrel{F_1\neq F_2}{F_1,F_2 < G}}t_{F_1}t_{F_2}. 
$$
Since $[K,L]_\PPP$ is semimodular, 
$$
\sum_G\sum_{\stackrel{F_1\neq F_2}{F_1,F_2 < G}}t_{F_1}t_{F_2}=\sum_{F_1\neq F_2}t_{F_1}t_{F_2}=\left(\sum_{F}t_F\right)^2-\sum_{F}t_F^2.
$$
Combining the equations above, we conclude 
$$
2 \ff_K^L= \left(\sum_{F}t_F\right)^2- \sum_{G} \left(t_G-\sum_{F<G}t_F\right)^2,
$$
which proves that the Hessian of $\ff_K^L$ has exactly one positive eigenvalue.
\end{proof}

\section{Log-concavity of the reduced characteristic polynomial}\label{lc}
The \emph{characteristic polynomial} of a matroid $\MM$ is 
\begin{equation}\label{chara}
\chi_{\MM}(t) = \sum_{F \in \LL(\MM)} \mu(\varnothing, F) t^{r(F,E)}. 
\end{equation}
If $\MM$ has rank at least one, then $\chi_\MM(t)$ is divisible by $t-1$, see \cite[Section 7]{White87}. The  \emph{reduced characteristic polynomial}  of a matroid $\MM$   is then 
$
\overline{\chi}_\MM(t) =  {\chi_\MM(t)}/(t-1)
$. Recall that a sequence $\{a_k\}_{k=0}^d$ is \emph{log-concave} if $a_k^2 \geq a_{k-1}a_{k+1}$ for all $0< k <d$. The next theorem, which was first proved by Adiprasito, Huh and Katz \cite{AHK}, solved the Heron-Rota-Welsh conjecture. We provide an alternative  proof below. 
\begin{theorem}\label{HRWconj}
The absolute values of the coefficients of the reduced characteristic polynomial of a matroid form a log-concave sequence. 
\end{theorem}

Consider the following two elements in the closure of $\SMOD_K^L$:
$$\alpha_K^L= \left(\frac {|S\setminus K|}{|L\setminus K|} \right)_{K\subset S \subset L} \mbox{ and } \ \ \beta_K^L = \left( \frac {|L\setminus S|}{|L\setminus K| } \right)_{K\subset S \subset L}.
$$
For $i \in L\setminus K$, let $\alpha_{K,i}^L=(a_S)_{K\subset S \subset L}$ and $\beta_{K,i}^L=(b_S)_{K\subset S \subset L}$ be the $0/1$-valued vectors defined by 
$a_S=1$ if and only if $i \in S$ and $b_S=1$ if and only if $i \not \in S$.
Then 
\begin{equation}\label{modeq}
\alpha_{K}^L-\alpha_{K,i}^L \in \MOD_K^L \ \ \mbox{ and } \ \ \beta_{K}^L-\beta_{K,i}^L \in \MOD_K^L,
\end{equation}
for all $i \in L \setminus K$. 

The elements $\alpha_K^L$ and $\beta_K^L$ behave well under the projections considered Section~\ref{posetsec}. We leave the proof to the reader. 

\begin{lemma}\label{alpha-beta}
If $K<F<L$, then  
$$
\pi_K^F(\alpha_K^L)= 0,  \ \ \ \pi_K^F(\beta_K^L) = \beta_K^F, \ \ \
\pi_F^L(\alpha_K^L)= \alpha_F^L,  \ \ \ \pi_F^L(\beta_K^L) = 0.
$$
\end{lemma}

\begin{lemma}\label{volalpha}
If  $\PPP$ is $1$-balanced, then 
$
\ff_K^L(\alpha_K^L) = 1 /{d(K,L)!}  
$
for all $K<L$.  
\end{lemma}
\begin{proof}
Let $i \in L\setminus K$, and let $\alpha_{K,i}^L=(a_S)_{K\subset S \subset L}$. Since $\PPP$ is $1$-balanced, there is a unique $H \in [K,L]_\PPP$ containing $i$ for which $K \prec H < L$. By \eqref{f-def}, \eqref{modeq} and Lemma~\ref{alpha-beta}, 
$$
d(K,L) \cdot \ff_K^L(\alpha_K^L)= \sum_{K<F<L} a_F \cdot  \ff_K^F ( 0)\cdot  \ff_F^L (\alpha_F^L) = \ff_H^L(\alpha_H^L), 
$$
from which the lemma follows by induction over $d(K,L)$. 
\end{proof}
Recall the M\"obius function of a locally finite poset \cite[Section 3.7]{stanley}.  The next theorem is usually stated as a consequence of Weisner's theorem, see \cite[p. 277]{stanley}. 
\begin{theorem}[Weisner's theorem]
If $x\prec a< y$ are elements in a semimodular lattice,  then 
$
\mu(x,y) = -\sum_b\mu(x,b) 
$
where the sum is over all $b$ for which $x<b\prec y$ and $a \not < b$. 
\end{theorem}
A consequence of Weisner's theorem is that the M\"obius function of a semimodular lattice $\LL$ alternates in sign, i.e., 
$(-1)^{\rho(b)-\rho(a)} \mu(a,b) \geq 0$, where $\rho$ is the rank function of $\LL$. 

\begin{lemma}\label{volbeta}
If $\PPP$ is the lattice of flats of a matroid and $K<L$ in $\PPP$, then
$
\ff_K^L(\beta_K^L)= {|\mu(K,L)|}/{d(K,L)!}. 
$

\end{lemma}

\begin{proof}
Let $i \in L\setminus K$, and let $\beta_{K,i}^L=(b_S)_{K\subset S \subset L}$. Then, by \eqref{f-def}, \eqref{modeq} and Lemma~\ref{alpha-beta}, 
$$
d(K,L) \cdot \ff_K^L(\beta_K^L)= \sum_{K<F<L} b_F  \cdot \ff_K^F(\beta_K^F) \cdot  \ff_F^L(0)= 
\sum_{\stackrel{F \not \ni i}{\tiny{K< F \prec L}}}\ff_K^F(\beta_K^F).
$$
The lemma now follows by induction and Weisner's theorem. 
\end{proof}

\begin{theorem}\label{comp}
Suppose $[K,L]_\PPP$ is the lattice of flats of a matroid. If  $i \in L\setminus K$, then 
$$
d(K,L)! \cdot \ff_K^L(s\alpha_K^L +t\beta_K^L) = \sum_{\stackrel {K \leq F < L} {i \not \in F}} \binom {d(K,L)}{r(K,F)}  \cdot |\mu(K,F)| \cdot t^{r(K,F)} s^{d(F,L)} . 
$$
\end{theorem}

\begin{proof}
Let $f(t)= \ff_K^L(s\alpha_K^L +t\beta_K^L)$. Then by Lemmas \ref{derivativen}, \ref{alpha-beta}, \ref{volalpha} and \ref{volbeta}, \eqref{modeq} and the chain rule, 
$$
f'(t)= \sum_{\stackrel {K < F < L} {i \not \in F}}   |\mu(K,F)| \cdot \frac {t^{d(K,F)}} {d(K,F)!} \frac {s^{d(F,L)}}{d(F,L)!}. 
$$
The lemma follows since by Lemma \ref{volalpha}, 
$
f(0)=  s^{d(K,L)} /{d(K,L)!}.
$
\end{proof}

\begin{lemma}[Cor. 7.27, \cite{White87}]\label{red}
If $i \in E$ is not a loop, then 
\begin{equation}\label{rcf}
\overline{\chi}_{\MM}(t) = \sum_{F \not \ni i} \mu(\varnothing,F) t^{d(F,E)}. 
\end{equation}
\end{lemma}

If $f$ is $\CCC$-Lorentzian and $\vv_1$ and $\vv_2$ are in the closure of  $\CCC$, then the bivariate polynomial 
$
f(s \vv_1+t\vv_2) = \sum_{k=0}^d \binom d k a_k  s^{d-k} t^{k}
$
is Lorentzian by Remark \ref{altdef} and the fact that the space of Lorentzian polynomials of degree $d$ is closed. By \cite[Example 2.26]{BH},  the sequence $\{a_k\}_{k=0}^d$ is \emph{log-concave}.

\begin{proof}[Proof of Theorem \ref{HRWconj}]
Let $b_k$  denote the absolute value of the coefficient in front of  $t^k$ in the reduced characteristic polynomial of $\MM$. Then 
$$
\sum_{k=0}^{r-1} \binom {r-1} k \cdot b_k \cdot s^{r-1-k}t^k = (r-1)! \cdot \ff_\varnothing^E (s\alpha_\varnothing^E +t\beta_\varnothing^E),
$$
by Theorem \ref{comp} and Lemma \ref{red}. The theorem now follows from Theorem \ref{mainChow} and the discussion above.  
\end{proof}

\section{The volume polynomial of the Chow ring of a matroid}\label{chowder}
Although it is not used in or proofs of Theorems \ref{mainChow} and \ref{HRWconj}, we prove in  this section that if $K<L$ are flats of a matroid $\MM$, then the polynomial $\ff_K^L$ is the volume polynomial of the Chow ring of the matroid obtained from $\MM$ by restricting to $L$, and then contracting $K$. 
The Chow ring of a matroid was introduced by Feichtner and Yuzvinsky in \cite{FY}, and  Adiprasito, Huh and Katz \cite{AHK} developed a Hodge theory for it. 
For our purposes it will be convenient to define the Chow ring of a nonempty open interval of the lattice of flats of a matroid, although this definition is easily seen to be equivalent to the Chow ring of a matroid as defined in \cite{AHK}. 

Let $K<L$ be flats of a matroid $\MM$. The \emph{Chow ring}, $A_K^L$ is defined as the quotient 
$$
\frac {\RR[x_F : K<F<L]} {I+J}, 
$$
where $I=I_K^L$ is the ideal generated by all quadratic monomial $x_Fx_G$, where $F$  and $G$ are any two non-comparable flats in $(K,L)$, and $J=J_K^L$ is the ideal generated by the linear forms
\begin{equation}\label{gener}
\sum_{F \ni i} x_F - \sum_{F \ni j} x_F, \mbox{ where $i$ and $j$ are any distinct elements of $L \setminus K$.} 
\end{equation}
The sums above are over $F$ such that $K<F<L$. 
The Chow ring is graded $A_K^L = \oplus_{k=0}^d (A_K^L)_k$, where $d=d(K,L)$, and where $(A_K^L)_d$ is isomorphic to $\RR$. Indeed, there is a well-defined isomorphism $\deg : (A_K^L)_d \to \RR$, 
\begin{equation}\label{dgr}
\deg(x_{F_1} x_{F_2} \cdots x_{F_d}) =1, \mbox{ for any flag } K\prec F_1 \prec F_2 \prec \cdots \prec F_d \prec L, 
\end{equation}
see \cite{AHK}. 

The \emph{volume polynomial}, $\vol_K^L(\ttt)$, of $A_K^L$ is defined by 
\begin{equation}\label{voldef}
\vol_K^L(\ttt)= \frac 1 {d!} \deg\left(\left( \sum_{K<F<L} x_Ft_F\right)^d\right), \ \ \ \mbox{ where } d=d(K,L).
\end{equation}

\begin{lemma}\label{mod-vol}
If $\ww \in \MOD_K^L$, then 
$\vol_K^L(\ttt+\ww)=\vol_K^L(\ttt)$ for all $\ttt \in \EE_K^L$.  
\end{lemma}
\begin{proof}
The space $\MOD_K^L$ is generated by vectors of the form $\alpha_{K,i}^L -\alpha_{K,j}^L$, $i,j \in L\setminus K$. Hence, by the definition of the ideal $J$, 
$$
\sum_{K<F<L} w_F x_F=0 \mbox{ in } A_K^L, \ \ \ \mbox{ for any } \ww=(w_S) \in \MOD_K^L. 
$$
Thus $\vol_K^L(\ttt+\ww)=\vol_K^L(\ttt)$, by \eqref{voldef}.
\end{proof}

\begin{lemma}\label{tens}
If $K < F <L$ are flats, then there is a ring homomorphism $\phi_F : A_K^F \otimes A_F^L \rightarrow A_K^L$, such that 
$$
\deg(\phi_F(\xi \otimes \eta)) =  \deg(\xi) \cdot \deg(\eta),
$$ 
for all $\xi \in A_K^F$ and $\eta \in A_F^L$.
\end{lemma}

\begin{proof}
First define  
$\phi_F : \RR[x_G : K<G<F] \otimes  \RR[x_G : F<G<L] \to A_K^L$,  by $\phi_F(\xi \otimes \eta) =\xi x_F \eta$.  
Then 
$$\xi \in I_K^F+J_K^F \mbox{ implies } \xi x_F \in  I_K^L+J_K^L, \mbox{ and } \eta \in I_F^L + J_F^L   \mbox{ implies } x_F\eta \in  I_K^L+J_K^L,$$
so that $\phi_F$ defines a homomorphism  $\phi_F : A_K^F \otimes A_F^L \rightarrow A_K^L$. 

The identity for the degrees now follows from \eqref{dgr}. 
\end{proof}

\begin{theorem}\label{chowVol}
For any two flats  $K < L$, $\vol_K^L= \ff_K^L$. 
\end{theorem}

\begin{proof}
We prove that $\vol_K^L$ satisfies the recursion \eqref{splitt}. Since $\vol_K^L= \ff_K^L$ whenever $d(K,L)=0$, the theorem will follow by induction and Euler's identity. 

Assume $d=d(K,L)>1$. Then 
$$
\frac \partial {\partial t_F} \vol_K^F(\ttt) \Big |_{t_F=0} = \frac 1 {(d-1)!} \sum_{k=0}^{d-1} \binom {d-1} k \cdot \deg (\xi^k x_F \eta^{d-1-k}), 
$$
where $\xi= \sum_{K<G<F} t_Gx_G$ and $\eta = \sum_{F<G<L} t_Gx_G$. If $k>d(K,F)$, then we have $\xi^k x_F\eta^{d-1-k}=0$ since $\phi_F$ is a homomorphism. Similarly $\xi^{d-1-j}x_F\eta^j =0$ if $j > d(F,L)$. Thus 
$$
\frac \partial {\partial t_F} \vol_K^L(\ttt) \Big |_{t_F=0}= \frac {\deg(\xi^{d(K,F)}x_F \eta^{d(F,L)})}{d(K,F)! \cdot d(F,L)!}=
\frac {\deg(\xi^{d(K,F)})}{d(K,F)!} \cdot \frac {\deg(\eta^{d(F,L)})}{d(F,L)!},
$$
by Lemma \ref{tens}. 
Hence 
$
\partial_{t_F} \vol_K^L(\ttt) \big |_{t_F=0}= \vol_K^F(\ttt) \cdot \vol_F^L(\ttt). 
$
Let $i \in F\setminus K$ and $j \in L\setminus F$. The $F$-entry of $\ttt-t_F(\alpha_{K,i}^L-\alpha_{K,j}^L)$ is zero. Hence by Lemma \ref{mod-vol}, 
$$
\frac \partial {\partial t_F} \vol_K^L(\ttt)=\frac \partial {\partial t_F} \vol_K^L(\ttt-t_F(\alpha_{K,i}^L-\alpha_{K,j}^L))= \vol_K^F(\ttt-t_F\alpha_{K,i}^L) \cdot \vol_F^L(\ttt-t_F\beta_{K,i}^L),
$$
which proves that $\vol_K^F(\ttt)$ satisfies the desired recursion. 
\end{proof}
\noindent 
\textbf{\large Acknowledgements}. 
The second author is grateful to Chris Eur and Mohan Ravichandran for
helpful discussions. The first author is a Wallenberg Academy Fellow supported by the Knut and Alice Wallenberg foundation and the G\"oran Gustafsson foundation. The second author is funded by the Deutsche Forschungsgemeinschaft (DFG, German Research Foundation) under Germany's Excellence Strategy -- The Berlin Mathematics Research Center MATH+ (EXC-2046/1, project ID: 390685689).

\end{document}